\documentclass[11pt,a4paper]{article}
 \usepackage{graphicx}
\usepackage{amsfonts}
\usepackage{amssymb}
\usepackage{amsthm}
\usepackage[centertags]{amsmath}
\usepackage{newlfont}
\usepackage{enumitem}
\usepackage{color}
\usepackage{graphicx,color}
\usepackage{amsmath, amssymb, graphics}

\def\R{\mathbb R}
 
\def\t{\mathbf t}
\def\n{\mathbf n}
\def\N{\mathbf N}
\def\b{\mathbf b}

\newtheorem{theorem}{Theorem}[section]

\newtheorem{definition}[theorem]{Definition}
\newtheorem{lemma}[theorem]{Lemma}
\newtheorem{remark}[theorem]{Remark}
\newtheorem{proposition}[theorem]{Proposition}

\title{Translation surfaces in Euclidean space with constant Gaussian curvature }
\author{Thomas Hasanis\\
Department of Mathematics\\
               University of Ioannina\\
               45110 Ioannina, Greece\\
\texttt{thasanis@cc.uoi.gr}
\and
Rafael L\'opez\footnote{Partially
supported by MEC-FEDER
 grant no. MTM2014-52368-P}\\
 Departamento de Geometr\'{\i}a y Topolog\'{\i}a\\ Instituto de Matem\'aticas (IEMath-GR)\\
 Universidad de Granada\\
 18071 Granada, Spain\\
\texttt{rcamino@ugr.es}}

\date{}
\begin{document}
\maketitle 

\begin{abstract}
We prove that the only surfaces in $3$-dimensional Euclidean space $\R^3$ with constant Gaussian curvature $K$ and constructed by the sum of two space curves  are cylindrical surfaces, in particular,  $K=0$.

\end{abstract}


\section{Introduction and statement of the result}

This paper  is concerned with the next problem in  classical differential geometry:
\begin{quote}{\it  What are the surfaces of Euclidean space $\R^3$ with constant Gaussian curvature that are the sum of two space curves?}
\end{quote}

The historical motivation of our problem comes from the classical text of G. Darboux \cite[Livre I]{da} where  the so-called \emph{surfaces d\'efinies par des propri\'et\'es cin\'ematiques} are considered, and later known as Darboux surfaces  in the literature. A Darboux surface is   defined kinematically as a  the movement of a curve by a uniparametric family of rigid motions of $\R^3$. Then  a parametrization of a such surface is $\Psi(s,t)=A(t)\alpha(s)+\beta(t)$ where $\alpha$ and $\beta$ are two space curves and $A(t)$ is an orthogonal matrix. In the case that we are considering in this paper, $A(t)$ is the identity. To be precise, we give the next definition: 
 
\begin{definition}
A surface $S\subset \R^3$ is called a   translation surface if it can be locally written as the sum $\Psi(s,t)=\alpha(s)+\beta(t)$ of two space curves  $\alpha:I\subset \R\rightarrow \R^3$ and $\beta:J\subset \R\rightarrow \R^3$. In the case where  $\alpha, \beta$ are plane curves lying on orthogonal planes, the surface is called a translation surface of plane type.
\end{definition}

The curves $\alpha$ and $\beta$ are called the generating curves of $S$.  Darboux deals with translation surfaces in Sects 81--84 \cite[pp. 137--142]{da}. The name of translation surface is because the surface obtained by the translation of $\alpha$ along $\beta$ (or {\it vici-versa} because the roles of $\alpha$ and $\beta$ are interchanged) and thus  all parametric curves $s=\mbox{const.}$ are congruent by translations   (similarly for parametric curves $t=\mbox{const.}$). 
 
For minimal surfaces,   Scherk proved  in 1835 that the only non-planar minimal  surface of type $z=f(x)+g(y)$ for two smooth functions $f$ and $g$   is   
$$    z(x,y)=\frac{1}{c}\log \left|\frac{\cos cy}{\cos cx}\right|,$$
where $c$ is a non-zero constant (\cite{sc}). A surface $z=f(x)+g(y)$ can be viewed as the sum of the planar curves $x\rightarrow (x,0,f(x))$ and $y\rightarrow (0,y,g(y))$, hence, the Scherk surface is the only non-planar minimal translation surface of plane type. Motivated by this example,  it is natural to ask what are the    translation surfaces that are minimal surfaces. It was proved in 1998 that when  one of the two generating curves of $S$  is planar, then $S$ is  the Scherk surface  (\cite{Dillen98}). More recently, the second author and O. Perdomo have characterized all minimal translation surfaces of $\R^3$ in terms of the curvature and torsion of the generating curves (\cite{lp}).

The problem of classification of   translation surfaces with constant Gaussian curvature $K$ is less known. A first example of a translation surface with constant Gaussian curvature $K=0$ is a cylindrical surface. Recall that a   cylindrical surface $S\subset\R^3$ is a ruled surface whose rulings are parallel to a constant direction. Then a cylindrical surface is a translation surface of plane type where one generating curve   is a  ruling and the other one   is a section of $S$ with a plane normal to the  rulings. Moreover, a cylindrical surface has  zero Gaussian curvature.  

The progress on  translation surfaces with constant Gaussian curvature $K$  has been as follows:

\begin{enumerate}
  \item If $S$ is a translation surface of plane type with constant Gaussian curvature $K$, then $K=0$ and $S$ is cylindrical (\cite{HLin}).
  \item The only translation surfaces with constant Gaussian curvature $K=0$ are cylindrical surfaces (\cite{Lopez2}).
  \item There are no translation surfaces with constant Gaussian curvature $K\neq0$ if one of the generating curves is a plane curve (\cite{Lopez2}).
\end{enumerate}

In the present paper we answer the initial problem and we classify all translation surfaces with constant Gaussian curvature. More precisely we prove:

\begin{theorem}\label{t1}
Cylindrical surfaces are the only translation surfaces in $\R^3$ with constant Gaussian curvature.
\end{theorem}

In the literature, there are many works   on the study of translation (hyper) surfaces  of plane type in different ambient spaces and different conditions on the curvatures, where the problem of finding such surfaces reduces into a problem of solving a PDE by separation of variables:  without to be a complete list, we refer: \cite{Lima, Lopez,lomu,mm,mp,Seo}.

The plan of our paper is as follows. In Sec. \ref{Prelim} we recall some known formulae on the local theory of curves and surfaces of $\R^3$ and we prove Th. \ref{t1}   in the particular case when one generating curve is a circle. In Sec. \ref{sec3} we give local conditions of the first and second fundamental forms so a metric is realizable in $\R^3$ as a metric of a translation surface. Also we give an alternative proof of Th. \ref{t1} in case that one curve is planar (see \cite{Lopez2}). Finally in Sec. \ref{sfin} we prove   Th. \ref{t1}. 

\section{Preliminaries}\label{Prelim}
For a general reference on curves and surfaces we refer to \cite{Carmo}. Moreover, the curves and surfaces considered will be assumed to be of class $C^\infty$.
Let $\alpha(s), s\in I$ and $\beta(t), t\in J$ be two curves in $\R^3$ parameterized by arc length  with curvatures $k_{\alpha}(s)>0, k_{\beta}(t)>0$, torsions $\tau_{\alpha}(s), \tau_{\beta}(t)$ and oriented Frenet trihedrons $\{{\t}_{\alpha}(s),{\n}_{\alpha}(s),{\b}_{\alpha}(s)\}$, $\{{\t}_{\beta}(t),{\n}_{\beta}(t),{\b}_{\beta}(t)\}$, for every $s\in I, t\in J$, respectively. In order to assume that $\alpha$ and $\beta$ are the generating curves of a regular translation surfaces $S\subset\R^3$, we suppose 
\begin{equation}\label{eq:2.1}
{\t}_\alpha(s)\times{\t}_\beta(t)\neq0,
\end{equation}
for all $(s,t)\in I\times J$, where $\times$ represents the vector product of  $\R^3$. Then  $\Psi(s,t)=\alpha(s)+\beta(t), (s,t)\in I\times J \subset \R^2$, is a parametrization of $S$. We have $\Psi_s={\t}_\alpha, \Psi_t={\t}_\beta$ and $\Psi_s\times \Psi_t\neq0$. Let $\phi(s,t)$, $0<\phi(s,t)<\pi$, be the angle that ${\t}_\alpha(s)$ makes with     ${\t}_\beta(t)$ at the point $\Psi(s,t)$, that is,  
$$\langle {\t}_\alpha(s),{\t}_\beta(t)\rangle =\cos\phi(s,t),$$
where $\langle,\rangle$ stands for the usual scalar product of $\R^3$.

\begin{remark}\label{remark:2.1}
 The parametric curves $t=\mbox{const.}$ are congruent and translations of $\alpha(s)$. Hence, they have the same curvature and torsion at corresponding points (similarly for the parametric curves $s=\mbox{const.}$)
\end{remark}
The first fundamental form of $S$ is $I=ds^2+2\cos\phi dsdt+dt^2$, 
and using Egregium Theorema we obtain the Gaussian curvature
\begin{equation}\label{eq:2.2}
K=-\frac{\phi_{st}}{\sin\phi},
\end{equation}
where $\phi_{st}$ is the second derivative of $\phi$ with respect to $s$ and $t$. Moreover, the unit normal ${\N}(s,t)$ of $S$ at the point $\Psi(s,t)$ is given by
\begin{equation}\label{eq:2.3}
{\N}(s,t)=\frac{{\t}_\alpha(s)\times{\t}_\beta(t)}{\sin\phi(t,s)}.
\end{equation}
Since $\Psi_{st}=0$, the second fundamental form of $S$ is $II=Ldu^2+Ndv^2$, where
$$L=\langle \Psi_{ss},{\N}\rangle=-\frac{k_\alpha}{\sin\phi}\langle {\b}_\alpha,{\t}_\beta\rangle,$$
$M=0$ and 
$$ N=\langle\Psi_{tt},{\N}\rangle =\frac{k_\beta}{\sin\phi}\langle {\t}_\alpha,{\b}_\beta\rangle.$$
These formulas are obtained by using the Frenet equations.

Having in mind the above analysis we consider the following more general  situation. Let $S\subset \R^3$ be a regular surface with parametrization $\Psi(u,v), (u,v)\in D\subset \R^2$, Gaussian curvature $K$ and fundamental forms
$$I=du^2+2\cos\phi\ dudv+dv^2,\quad II=Ldu^2+Ndv^2,$$
where $\phi(u,v)$ is the angle between parametric curves at $\Psi(u,v)$ and $0<\phi(u,v)<\pi$.
The function $\phi(u,v)$ is differentiable. For the Gaussian curvature $K$ we have
\begin{equation}
K=\frac{LN}{\sin^2\phi},
\label{eq:2.4}
\end{equation}
and by Egregium Theorema
\begin{equation}\label{eq:2.5}
\phi_{uv}=-K\sin\phi.
\end{equation}
Moreover, the Codazzi equations for $S$ becomes

\begin{equation}\label{eq:2.6}
L_v=\frac{N}{\sin\phi}\phi_u,\quad N_u=\frac{L}{\sin\phi}\phi_v.
\end{equation}
In order to compute the curvature and torsion of parametric curves $\alpha(u)=\Psi(u,const.)$ and $\beta(v)=\Psi(const.,v)$ we need the following:
\begin{enumerate}
  \item the Christoffel symbols
  \begin{eqnarray*}
      \Gamma^{1}_{11}&=\dfrac{\cos\phi}{\sin\phi}\phi_u, &\Gamma^{2}_{11}=-\frac{\phi_u}{\sin\phi}, \\
      \Gamma^{1}_{12}&=\Gamma^{1}_{21}=0, & \Gamma^{2}_{12}=\Gamma^{2}_{21}=0, \\
      \Gamma^{1}_{22}&=-\dfrac{\phi_v}{\sin\phi}, &\Gamma^{2}_{22}=\frac{\cos\phi}{\sin\phi}\phi_v. \\
    \end{eqnarray*}

  \item the Gauss formulas
\begin{eqnarray*}
\Psi_{uu}&=&\frac{\phi_u}{\sin\phi}(\cos\phi \Psi_u-\Psi_v)+L{\N}, \\
\Psi_{uv}&=&\Psi_{vu}=0,\\
\Psi_{vv}&=&\frac{\phi_v}{\sin\phi}(\Psi_u-\cos\phi \Psi_v)+N{\N},
\end{eqnarray*}
   
  \item and the Weingarten formulas
\begin{align*}
{\N}_u=-\frac{L}{\sin^2\phi}\Psi_u+\frac{L\cos\phi}{\sin^2\phi}\Psi_v,\\
{\N}_v=\frac{N\cos\phi}{\sin^2\phi}\Psi_u-\frac{N}{\sin^2\phi}\Psi_v.
\end{align*}
\end{enumerate}
The curvature $k_\alpha$ of the parametric curve $\alpha(u)$, which does not depend on $v$, is given by
$$k_\alpha=\left| \alpha'' \right|=\left| \Psi_{uu} \right| = (L^2+\phi^2_u)^{1/2},$$
where the prime ($'$) denotes derivative with respect to $u$. Moreover, the torsion $\tau_\alpha$ of $\alpha(u)$ at points where   $k_\alpha(u)>0$ is given by
$$\tau_\alpha=\frac{(\alpha',\alpha'',\alpha''')}{\left|\alpha'\times\alpha''\right|^2}=\frac{(\Psi_u,\Psi_{uu},\Psi_{uuu})}{\left|\Psi_u\times \Psi_{uu}\right|^2},$$
where ($\cdot,\cdot,\cdot$) is the mixed product in $\R^3$. Using the Gauss and Weingarten formulas, we obtain
\begin{equation}\label{eq:2.7}
\tau_\alpha=\bigg(L\phi_{uu}-\frac{L\cos\phi}{\sin\phi}(L^2+\phi^2_u)-L_u\phi_u\bigg)\big(L^2+\phi^2_u\big)^{-1}.
\end{equation}
Analogously, the curvature $k_\beta$ and torsion $\tau_\beta$ of the parametric curve $\beta(v)$ are
$$k_\beta=(N^2+\phi^2_v)^{1/2},$$
$$\tau_\beta=\bigg(N\phi_{vv}-\frac{N\cos\phi}{\sin\phi}(N^2+\phi^2_v)-N_v\phi_v\bigg)\big(N^2+\phi^2_v\big)^{-1}.$$

For later use we set
\begin{equation} \label{eq:2.8}
        A  = (L^2+\phi^2_u)^{1/2},\qquad
        B  = (N^2+\phi^2_v)^{1/2},
\end{equation}
where  $A, B$ are non-negative functions of one variable $u$ and $v$, respectively. Moreover, $A, B$ represent the curvatures of parametric curves $\alpha(u)$ and $\beta(v)$.

In the next lemma we obtain a formula, useful in the proof of   Th. \ref{t1}, for the torsion of parametric curves in the case where the Gaussian curvature of $S$ is non-zero everywhere.
\begin{lemma}\label{lemma:2.1}
Let $S\subset \R^3$ be a  surface with parametrization $\Psi(u,v)$, $(u,v)\in D\subset \R^2$, and fundamental forms
\begin{eqnarray*}
I&=&du^2+2\cos\phi\ dudv+dv^2,\ 0<\phi<\pi\\
II&=&Ldu^2+Ndv^2.
\end{eqnarray*}
If the Gaussian curvature  $K$ is non-zero everywhere on $D$, then
\begin{equation}
\phi_{uu}=\tau(u)L+\frac{\cos\phi}{\sin\phi}L^2
+\frac{A'}{A}\phi_u
\label{eq:2.9}
\end{equation}
and
\begin{equation}
\phi_{vv}=\tau(v)N+\frac{\cos\phi}{\sin\phi}N^2
+\frac{B'}{B}\phi_v,
\label{eq:2.10}
\end{equation}
where $\tau(u), \tau(v)$ are the torsions of the parametric curves  $\alpha(u), \beta(v)$, respectively.
\end{lemma}

\begin{proof}\label{proof}
 Because of $K\neq0$, equation (\ref{eq:2.4}) implies $L\neq0$ and $N\neq0$. Thus, the curvatures of parametric curves are non-zero and $A>0$. From $A^2=L^2+\phi^2_u$, we obtain $L^2=A^2-\phi^2_u$ and $AA'=LL_u+\phi_u\phi_{uu}$. Taking into account these relations, from (\ref{eq:2.7}) we conclude the desired relation (\ref{eq:2.9}). In a similar way we obtain (\ref{eq:2.10}).
\end{proof}For later use we prove the following

\begin{lemma}\label{lemma:2.2}
Let $S\subset \R^3$ be a regular translation surface with constant Gaussian curvature $K$. If one generating curve is a circle, then   $S$ is cylindrical and  $K=0$.
\end{lemma}

\begin{proof}\label{proof2}
 Without loss of generality, we suppose that $\alpha$ and $\beta$ are parametrized by the arc-length and $\alpha$ is a circle of radius $r>0$ in the $xy$-plane. Then a parametrization of $S$ is  
$$\Psi(s,t)=(r\cos(s/r),r\sin(s/r),0)+(\beta_1(t),\beta_2(t),\beta_3(t)).$$
In this situation, Equation (\ref{eq:2.4}) becomes
$$K=\frac{(\Psi_s,\Psi_t,\Psi_{ss})\cdot(\Psi_s,\Psi_t,\Psi_{tt})}{\big(1-(-\beta_1'\sin(s/r) +\beta_2'\cos(s/r))^2\big)^2},$$
where prime ($'$) denotes derivative with respect to $t$. We compute $\Psi_s,\Psi_t, \Psi_{ss}, \Psi_{tt}$ and insert in the last equation. Since the functions $1$, $\sin(s/r)$, $\cos(s/r)$,$ \sin(2s/r)$, $\cos(2s/r)$, $\sin(4s/r)$ and $\cos(4s/r)$ are linearly independent we take inter alia the following two relations
\begin{equation}
\beta_3'(\beta_2'\beta_3''-\beta_2''\beta_3')=\beta_3'(\beta_1'\beta_3''-\beta_1''\beta_3')=0
\label{eq:2.11}
\end{equation}
\begin{equation}
K\bigg(\Big(1-\frac{(\beta_1')^2+(\beta_2')^2}{2}\Big)^2+\frac{(\beta_1'\beta_2')^2}{2}+\frac{1}{2}\Big(\frac{(\beta_1')^2-(\beta_2')^2}{2}\Big)^2\bigg)=0.
\label{eq:2.12}
\end{equation}

From (\ref{eq:2.12}) it follows that $K=0$. Furthermore, (\ref{eq:2.11}) gives two possibilities: 
\begin{enumerate}
\item Case $\beta_3'=0$  for every $t$. Then $S$ is (part of) a (horizontal) plane.
\item Case $\beta_3'\not=0$ at some $t=t_0$. Then locally we have $\beta_2'\beta_3''-\beta_2''\beta_3'= \beta_1'\beta_3''-\beta_1''\beta_3'=0$. Then it is immediate that $\beta$ parametrizes a straight-line which is not contained in the $xy$-plane. This proves that $S$ is a cylindrical surface.
\end{enumerate}

\end{proof}

\section{A particular case of Theorem \ref{t1}}\label{sec3}
In the proof of   Th. \ref{t1}, we will make use of the following proposition which is interesting in itself.

\begin{proposition}\label{proposition:3.1}
Let $\Psi:D\subset\R^2\rightarrow \R^3$, $\Psi=\Psi(u,v)$,  be a parametrization of a regular   surface $S$ with  Gaussian curvature $K$. Suppose that the fundamental forms are
$$I=du^2+2\cos\phi\ dudv+dv^2,\quad II=Ldu^2+Ndv^2,$$
where   $0<\phi(u,v)<\pi$  is the angle between the parametric curves. Then, locally, we have the following assertions:
\begin{enumerate}
    \item There are two non-negative functions $A(u), B(v)$, with $A^2\geq\phi^2_u$ and $B^2\geq\phi^2_v$, such that
    \begin{equation}
        L^2+\phi^2_u=A^2, N^2+\phi^2_v=B^2,
        \label{eq:3.1}
        \end{equation}
        and
        \begin{equation}
        (A^2-\phi^2_u)(B^2-\phi^2_v)=K^2\sin^4\phi.
        \label{eq:3.2}
        \end{equation}
    \item If $K=0$ everywhere on $D$, then $S$ is cylindrical and a translation surface.
  \item If $K\neq0$ everywhere on $D$, then $S$ is translation surface with $\Psi(u,v)=\alpha(u)+\beta(v)$ and where  the curves $\alpha(u), \beta(v)$ in $\R^3$ are parametrized by the arc-length with curvatures and torsions
\begin{eqnarray}\label{eq:3.3}
  k_\alpha &=& \sqrt{L^2+\phi^2_u}=A>0 \\
\nonumber  \tau_\alpha &=& \varepsilon_1(\phi_{uu}-\frac{\cos\phi}{\sin\phi}(A^2-\phi^2_u)-\frac{A'}{A}\phi_u)
(A^2-\phi^2_u)^{1/2} \\
\nonumber  k_\beta &=& \sqrt{N^2+\phi^2_v}=B>0 \\
\nonumber  \tau_\beta &=& \varepsilon_2(\phi_{vv}-\frac{\cos\phi}{\sin\phi}(B^2-\phi^2_v)-\frac{B'}{B}\phi_v)
(B^2-\phi^2_v)^{1/2},
\end{eqnarray}
where $\varepsilon_1=sign(L)=\pm1$ and $\varepsilon_2=sign(N)=\pm1$.
\item Let $ds^2=du^2+2\cos\phi+dv^2$ a metric on $D\subset \R^2$ with $0<\phi(u,v)<\pi$ and  Gaussian curvature $K\neq0$ everywhere on $D$. If there exist two smooth non negative functions $A(u), B(v)$ which satisfy (\ref{eq:3.2}) and \mbox{$A^2>\phi^2_u, B^2>\phi^2_v$}, then the metric is realizable in $\R^3$ as a metric of a translation surface with
    $$ L=\varepsilon_1(A^2-\phi^2_u)^{1/2}, M=0, N=\varepsilon_2(B^2-\phi^2_v)^{1/2},$$
where $\varepsilon_1, \varepsilon_2=\pm1$ and $\varepsilon_1\varepsilon_2=sign(K)$.
\end{enumerate}
\end{proposition}

\begin{proof}
    \begin{enumerate}
        \item By taking into account equation (\ref{eq:2.4}), the Codazzi equations (\ref{eq:2.6}) give
   $$               LL_v=\frac{LN}{\sin\phi}\phi_u,$$
   that is,
   $$ \frac{1}{2}\frac{\partial L^2}{\partial v}=K \phi_u \sin\phi.$$
            Now, the Egregium Theorema implies
            $$  \frac{1}{2}\frac{\partial L^2}{\partial v}=-\phi_{uv}\phi_{u}=-\frac{1}{2}\frac{\partial}{\partial v}\phi^2_u$$
            or equivalently,
             $$ \frac{\partial}{\partial v}(L^2+\phi^2_u)=0.  $$
            Hence, there is a non-negative function $A(u)$ such that $L^2+\phi^2_u=A^2(u)$. In a similar way we obtain the second of equation (\ref{eq:3.1}). Combining Eqs. (\ref{eq:2.4}) and (\ref{eq:3.1})  we conclude (\ref{eq:3.2}).

        \item Because of locality we consider two cases. If  $L=N=0$ everywhere on $D$ then the surface is a plane and so, cylindrical. If $L\neq0$ and $N=0$ everywhere on $D$, because of $K=0$, then from the second of (\ref{eq:2.6}) we take $\phi_v=0$ everywhere. The Gauss formulas imply $\Psi_{uv}=\Psi_{vv}=0$. So, by a double integration we get $\Psi(u,v)=\beta(u)+v\textbf{a}$, where $\textbf{ a}$ is a constant vector with norm $|\textbf{a}|=1$. Thus, $S$ is cylindrical and hence translation surface.

        \item Since $\Psi_{uv}=0$, the second of Gauss formulas gives $\Psi(u,v)=\alpha(u)+\beta(v)$ and hence $S$ is a translation surface. Lemma 2.1 and the preceding analysis gives (\ref{eq:3.3}).

    \item It is enough to show that the coefficients $E=1, F=\cos\phi, G=1$ of the given metric and $L=\varepsilon_1(A^2-\phi^2_u)^{1/2}, M=0, N=\varepsilon_2(B^2-\phi^2_v)^{1/2}$ satisfy the Gauss equation (\ref{eq:2.4}) and the Codazzi equations (\ref{eq:2.6}). The Gauss equation holds because of the assumption (\ref{eq:3.2}). By differentiating   $L^2=A^2-\phi^2_u$ with respect to $v$, we have
            $$LL_v=-\phi_u\phi_{uv}=-\phi_u(-K\sin\phi)=\frac{LN}{\sin^2\phi}\phi_u\sin\phi=\frac{LN\phi_u}{\sin\phi},$$
that is,
  $$ L\bigg(L_v-\frac{N\phi_u}{\sin\phi}\bigg)=0, $$
  from which, because of $K\neq0$, we obtain the first of Codazzi equations. In a similar way we obtain the second of (\ref{eq:2.6}).
 \end{enumerate}
\end{proof}

\begin{remark}\label{remark:3.1} The preceding analysis and Prop. 3.1 re-establish the following result: a regular surface $S\subset \R^3$ is a translation surface if and only if it has a parametrization $\Psi(u,v)$ with fundamental forms
$$I=du^2+2\cos\phi\ dudv+dv^2,\quad II=Ldu^2+Ndv^2.$$
\end{remark}

In our process to prove Th. \ref{t1}, we consider a translation surface $S\subset\R^3$ with   non zero constant Gaussian curvature $K$. After a similarity of the ambient space, we can suppose that $K$ is $-1$ or $1$.  In what follows, this hypothesis will be assumed without further comment. Let  $\Psi(u,v)$, $(u,v)\in D\subset \R^2$, be a parametrization of $S$   and fundamental forms
$$I=du^2+2\cos\phi\ dudv+dv^2,\quad II=Ldu^2+Ndv^2,$$
where $\phi(u,v)$  is the angle of parametric curves and $0<\phi(u,v)<\pi$. The torsion of parametric curves $v=const.$ is  denoted by $\tau=\tau(u)$. Moreover, we choose the orientation of $S$ so that $L>0$ everywhere. Then, from   (\ref{eq:3.3}) we have
\begin{equation}
\phi_{uu}=\tau X^{1/2}+\frac{\cos\phi}{\sin\phi}X+\frac{A'}{A}\phi_u,
\label{eq:3.4}
\end{equation}
where we have set $X=L^2=A^2-\phi^2_u$. Since $K$ is 1 or -1 we have $X>0$ and thus $A>0$. Moreover, because we are interested for local solutions of the problem, we may suppose that $\phi_u\neq0$ and $\phi_v\neq0$ everywhere. Indeed, otherwise we have $K=0$ from Eq. (\ref{eq:2.5}).
We postpone for a little the proof of Th. \ref{t1} inserting here a result which has been proved in \cite{Lopez2}. For completeness, we give a proof with another method which will guide us during the proof of the general case in Sec. \ref{sfin}.

\begin{proposition}\label{proposition:3.2}
Let $S\subset \R^3$ be a regular translation surface with   constant Gaussian curvature $K$. If one of the parametric curves is a plane curve, then $K=0$ and $S$ is cylindrical.
\end{proposition}

\begin{proof}\label{proof:3.2}
    If $K$ is non-zero, up to a similarity, we can suppose that $K=1$ or $K=-1$. Without loss of generality, we assume that the parametric curves $v=const.$ are plane curves and thus $\tau(u)=0$. Then Eq. (\ref{eq:3.4}) gives
    \begin{equation}\label{eq:3.5}
      \phi_{uu}=\frac{\cos\phi}{\sin\phi}X+\frac{A'}{A}\phi_u,
    \end{equation}
    where $X=A^2-\phi^2_u>0$ and $A>0$. By differentiating   (\ref{eq:3.5}) with respect to $v$ and taking into account   (\ref{eq:2.5}) we obtain
    $$ (-K\sin\phi)_u=-\frac{1}{\sin^2\phi}\phi_vX+\frac{\cos\phi}{\sin\phi}(-2\phi_u\phi_{uv})+\frac{A'}{A}\phi_{uv},$$
  or equivalently,
  $$  3K\phi_u\cos\phi=K\frac{A'}{A}\sin\phi+\frac{1}{\sin^2\phi}\phi_vX.$$
 Multiplying by $K\sin^2\phi$ we find
 \begin{equation}\label{eq:3.6}
      K\phi_{v}=(3\phi_u\cos\phi \sin^2\phi-\frac{A'}{A}\sin^3\phi)X^{-1}.
    \end{equation}
 By differentiating   (\ref{eq:3.6}) with respect to $u$ and taking into account relations (\ref{eq:2.5}) and (\ref{eq:3.5}) we get
    \begin{eqnarray*}
      -\sin\phi&=&\bigg(-3\phi_u^2\sin^3\phi+6\phi^2_u\cos^2\phi \sin\phi+3\cos^3\phi \sin\phi X\\
      && + 3\frac{A'}{A}\phi_u\cos\phi \sin^2\phi-\big(\frac{A'}{A}\big)'\sin^3\phi-3\frac{A'}{A}\phi_u\sin^2\phi \cos\phi \bigg)X^{-1}\\
      && - \bigg((3\phi_u\cos\phi \sin^2\phi -\frac{A'}{A}\sin^3\phi)\Big(\frac{2A'}{A}(A^2-\phi^2_u)-2\frac{\cos\phi}{\sin\phi}\phi_uX\Big)\bigg)X^{-2},
    \end{eqnarray*}
 or equivalently 
 \begin{multline*}
        (1+3\cos^2\phi)X+\bigg(12\phi_u^2\cos^2\phi -3\phi_u^2\sin^2\phi -8\frac{A'}{A}\phi_u \cos\phi \sin\phi  \\
      -\big(\frac{A'}{A}\big)'\sin^2\phi+2\big(\frac{A'}{A}\big)^2\sin^2\phi\bigg)=0.
    \end{multline*}
 We write this expression in terms of the angle $2\phi$,   obtaining
  \begin{equation}\label{eq:3.7}
      -\frac{4A'}{A}\phi_u \sin2\phi+\bigg(\frac{3}{2}X+\frac{15}{2}\phi^2_u-\frac{\Sigma}{2}\bigg)\cos2\phi+\bigg(\frac{5}{2}X+\frac{9}{2}\phi^2_u+\frac{\Sigma}{2}\bigg)=0,
    \end{equation}
where $\Sigma=2(A'/A)^2-(A'/A)'.$
 Equation (\ref{eq:3.7}) is of the form
 \begin{equation}\label{eq:3.8}
      P\sin2\phi+Q\cos2\phi+R=0,
    \end{equation}
    where
    \begin{eqnarray*}
        P&=&-4\frac{A'}{A}\phi_u\\
        Q&=&\frac{3}{2}X+\frac{15}{2}\phi^2_u-\frac{\Sigma}{2}\\
        R&=&\frac{5}{2}X+\frac{9}{2}\phi^2_u+\frac{\Sigma}{2}.
    \end{eqnarray*}
 By differentiation   (\ref{eq:3.8})  with respect to $u$, we obtain
 \begin{equation}\label{eq:3.9}
      (P'-2Q\phi_u)\sin2\phi+(Q'+2P\phi_u)\cos2\phi+R'=0,
    \end{equation}
 where
    \begin{align*}
      P'-2Q\phi_u&=-4\frac{A'}{A}X\frac{\cos\phi}{\sin\phi}-4\big(\frac{A'}{A}\big)'\phi_u-4\big(\frac{A'}{A}\big)^2\phi_u-3X\phi_u-15\phi^3_u+\Sigma\phi_u,\\
      Q'+2P\phi_u&=12X \phi_u \frac{\cos\phi}{\sin\phi}+3AA'+12\frac{A'}{A}\phi^2_u-\frac{\Sigma'}{2}-8\frac{A'}{A}\phi^2_u,\\
    R'&=4 X\phi_u\frac{\cos\phi}{\sin\phi}+5AA'+4\frac{A'}{A}\phi^2_u+\frac{\Sigma'}{2},
    \end{align*}
and the prime denotes derivative with respect to $u$. Substituting the above as well as $\cos\phi/\sin\phi= \sin2\phi/(1-\cos2\phi)$ in Eq.  (\ref{eq:3.9}) we find
\begin{eqnarray*}
      &&\left(-4\big(\frac{A'}{A}\big)'\phi_u-4\big(\frac{A'}{A}\big)^2\phi_u+\phi_uX-15\phi^3_u+\Sigma\phi_u\right)\sin2\phi \\
      &+&\left(4\big(\frac{A'}{A}\big)'\phi_u+4\big(\frac{A'}{A}\big)^2\phi_u+15\phi_uX+15\phi^3_u-\Sigma\phi_u\right)\sin2\phi \cos2\phi \\
      &+&\left(-3AA'-4\frac{A'}{A}\phi^2_u+\frac{\Sigma'}{2}+4\frac{A'}{A}X\biggl)\cos^22\phi-\biggl(2AA'+\Sigma'\right)\cos2\phi  \\
      &+&\left(5AA'+4\frac{A'}{A}\phi^2_u+\frac{\Sigma'}{2}-4\frac{A'}{A}X\right)=0.
    \end{eqnarray*}
 By inserting   $X=A^2-\phi^2_u$ and $\Sigma=2( A'/A)^2-( A'/A)'$ in this expression, we get
    \begin{eqnarray}\label{eq:3.10}
      &&\left(-16\phi^3_u+\Big(A^2-5\big(\frac{A'}{A}\big)'-2\big(\frac{A'}{A}\big)^2\Big)\phi_u\right)\sin2\phi\\
      \nonumber&&+\left(15A^2 +5\big(\frac{A'}{A}\big)'+2\big(\frac{A'}{A}\big)^2\right)\phi_u\sin2\phi \cos2\phi\\
      \nonumber&&+\left(-8\frac{A'}{A}\phi_u^2+AA'+\frac{\Sigma'}{2}\right)\cos^22\phi-(2AA'+\Sigma')\cos2\phi \\
      \nonumber&&+\left(-8\frac{A'}{A}\phi_u^2+AA'+\frac{\Sigma'}{2}\right)=0.
    \end{eqnarray}
    Equation (\ref{eq:3.7}), by inserting $X=A^2-\phi^2_u$, becomes

\begin{equation}\label{eq:3.11}
      \frac{4A'}{A}\phi_u\sin2\phi=\bigg(6\phi_u^2+\frac{3}{2}A^2-\frac{\Sigma}{2}\bigg)\cos2\phi+\bigg(2\phi_u^2+\frac{5}{2}A^2+\frac{\Sigma}{2}\bigg).
    \end{equation}

    Multiplying (\ref{eq:3.10}) by $4A'/A$ and taking into account (\ref{eq:3.11}) we obtain an equation of the form
    \begin{equation}\label{eq:3.12}
        b_2\cos^22\phi+b_1\cos2\phi+b_0=0,
    \end{equation}    
    with
    \begin{eqnarray*}
b_2&=&b_{22}\phi^2_u+b_{20}=\left(90A^2+30\big(\frac{A'}{A}\big)'-20\big(\frac{A'}{A}\big)^2\right)\phi^2_u \\
&&+\left(15A^2+5\big(\frac{A'}{A}\big)'+2\big(\frac{A'}{A}\big)^2\right)\left(\frac{3}{2}A^2-\frac{\Sigma}{2}\right)+4(A')^2+2\frac{A'\Sigma'}{A},\\
b_1&=& b_{14}\phi^4_u+b_{12}\phi^2_u+b_{10}=-96\phi^4_u\\
&&+\left(12A^2-20\big(\frac{A'}{A}\big)'-8\big(\frac{A'}{A}\big)^2+8\Sigma\right)\phi^2_u\\
&&+\left(A^2-5\big(\frac{A'}{A}\big)'-2\big(\frac{A'}{A}\big)^2\right)\bigg(\frac{3}{2}A^2-\frac{\Sigma}{2}\bigg)\\
&&+\bigg(15A^2+5\big(\frac{A'}{A}\big)'+2\big(\frac{A'}{A}\big)^2\bigg)\bigg(\frac{5}{2}A^2+\frac{\Sigma}{2}\bigg)-8(A')^2-4\frac{A'\Sigma'}{A},\\
b_0&=&b_{04}\phi^4_u+b_{02}\phi^2_u+b_{00}\\
&=&-32\phi^4+\left(-38A^2-10\big(\frac{A'}{A}\big)'+28\big(\frac{A'}{A}\big)^2-8\Sigma\right)\phi^2_u\\
&&+\left(A^2-5\big(\frac{A'}{A}\big)'-8\big(\frac{A'}{A}\big)^2\right)\left(\frac{5}{2}A^2+\frac{\Sigma}{2}\right)+4(A')^2+2\frac{A'\Sigma'}{A}.
\end{eqnarray*}
  From Eq.  (\ref{eq:3.11}) a raising squared we get one equation of the form
    \begin{equation}\label{eq:3.13}
        c_2\cos^22\phi+c_1\cos2\phi+c_0=0,
    \end{equation}
    with
    \begin{eqnarray*}
&c_2=c_{24}\phi_u^4+c_{22}\phi_u^2+c_{20}= 36\phi_u^4 +\left(18A^2-6\Sigma+16\big(\frac{A'}{A}\big)^2\bigg)\phi_u^2+\bigg(\frac{3}{2}A^2-\frac{\Sigma}{2}\right)^2\\
&c_1=c_{14}\phi_u^4+c_{12}\phi_u^2+c_{10}= 24\phi_u^4 +\left(36A^2+4\Sigma\bigg)\phi_u^2+\bigg(3A^2-\Sigma\bigg)\bigg(\frac{5}{2}A^2+\frac{\Sigma}{2}\right)\\
&c_0=c_{04}\phi_u^4+c_{02}\phi_u^2+c_{00}=4\phi_u^4 +\left(10A^2+2\Sigma-16\big(\frac{A'}{A}\big)^2\bigg)\phi_u^2+\bigg(\frac{5}{2}A^2+\frac{\Sigma}{2}\right)^2.
    \end{eqnarray*}
    The two quadratic equations (\ref{eq:3.12}) and (\ref{eq:3.13}) have at least a common solution with respect to $\cos2\phi$. Thus, their eliminant must be zero, hence we have
\begin{equation}\label{eq:3.14}
        (b_2c_0-b_0c_2)^2-(b_0c_1-b_1c_0)(b_1c_2-b_2c_1)=0.
    \end{equation}
  In the meantime, we have compute each one of the above factor only in the higher coefficients with respect to $\phi_u$:
 \begin{eqnarray*}
      b_2c_0-b_0c_2 &=& (b_{22}\phi^2_u+b_{20})(c_{04}\phi^4_u+c_{02}\phi^2_u+c_{00}) \\
       &&-(b_{04}\phi^4_u+b_{02}\phi^2_u+b_{00})(c_{24}\phi^4_u+c_{22}\phi^2_u+c_{20}) \\
       &=& -b_{04}c_{24}\phi^8_u+(b_{22}c_{04}-b_{04}c_{22}-b_{02}c_{24})\phi^6_u +\ldots\\
       &=& 1152\phi^8_u+384\bigg(6A^2+(\frac{A'}{A})'-(\frac{A'}{A})^2\bigg)\phi^6_u+\ldots\\
       &=& 384\bigg(3\phi^8_u+\Big(6A^2+(\frac{A'}{A})'-(\frac{A'}{A})^2\Big)\phi^6_u\bigg)+\ldots\\
      b_0c_1-b_1c_0 &=& (b_{04}\phi^4_u+b_{02}\phi^2_u+b_{00})(c_{14}\phi^4_u+c_{12}\phi^2_u+c_{10})\\
      &&-(b_{14}\phi^4_u+b_{12}\phi^2_u+b_{10})(c_{04}\phi^4_u+c_{02}\phi^2_u+c_{00}) \\
       &=& (b_{04}c_{14}-b_{14}c_{04})\phi^8_u+(b_{04}c_{12}+b_{02}c_{14}-b_{14}c_{02}-b_{12}c_{04})\phi^6_u +\ldots\\
       &=& -384\bigg(\phi^8_u+3\Big(A^2+(\frac{A'}{A})^2\Big)\phi^6_u\bigg)+\ldots\\
      b_1c_2-b_2c_1 &=& (b_{14}\phi^4_u+b_{12}\phi^2_u+b_{10})(c_{24}\phi^4_u+c_{22}\phi^2_u+c_{20})\\
      &&-(b_{22}\phi^2_u+b_{20})(c_{14}\phi^4_u+c_{12}\phi^2_u+c_{10}) \\
       &=& b_{14}c_{24}\phi^8_u+(b_{14}c_{22}+b_{12}c_{24}-b_{22}c_{14})\phi^6_u +\ldots\\
       &=& -384\bigg(9\phi^8_u+\Big(9A^2-(\frac{A'}{A})^2+6(\frac{A'}{A})'\Big)\phi^6_u\bigg)+\ldots.
    \end{eqnarray*}
Equation (\ref{eq:3.14}) is a polynomial equation with respect to $\phi_u$ at most of degree sixteen, and its coefficients depend only on the variable $u$. We claim that all coefficients must be zero: otherwise the equation has a solution of the form $\phi_u=f(u)$, but this is impossible because in that case and from (\ref{eq:2.5}), we have $K=0$, which is a contradiction. The coefficient of $\phi^{16}_u$ is trivially zero, but the coefficient of $\phi^{14}_u$ is $-14( A'/A)^2$. Hence we have $A'=0$, that is, $A$ is a positive constant.  So the parametric curves $v=const.$ are circles. By  Lemma \ref{lemma:2.2} we deduce that $K=0$, a contradiction and the proof of proposition is completed.
\end{proof}

\section{Proof of Theorem \ref{t1}}\label{sfin}

We return now to the proof of Th. \ref{t1}. In Prop. \ref{proposition:3.2} we have proved that if one generating curve is a planar curve, then $S$ is a cylindrical surface. Then the proof of Th. \ref{t1} is by contradiction assuming that none of the generating curves is planar, so the torsions $\tau=\tau(u)$ and $\tau=\tau(v)$ are non zero everywhere.
Differentiating (\ref{eq:3.4}) with respect to $v$ and taking into account (\ref{eq:2.5}) we get
\begin{equation}\label{eq:3.15}
  K\phi_v
=(3\phi_u\cos\phi \sin^2\phi -\frac{A'}{A}\sin^3\phi)X^{-1}+\tau \phi_u\sin^3\phi X^{-3/2}.
\end{equation}
Moreover,    differentiating (\ref{eq:3.15}) with respect to $u$ and using (\ref{eq:3.4}), we obtain
\begin{eqnarray*}
  &&(1+3\cos^2\phi)X^2 +4\tau \cos\phi \sin\phi X^{3/2}+\biggl(12\phi^2_u\cos^2\phi-3\phi^2_u\sin^2\phi    \\
  &&-8\frac{A'}{A}\phi_u\cos\phi \sin\phi+\biggl(2(\frac{A'}{A})^2-(\frac{A'}{A})'+\tau^2\biggl)\sin^2\phi\biggl)X  \\
  &&+\biggl(12\tau\phi_u^2 \cos\phi \sin\phi -4\tau(\frac{A'}{A})\phi_u\sin^2\phi +\tau'\phi_u\sin^2\phi\biggl)X^{1/2}  \\
  &&+3\tau^2\phi^2_u\sin^2\phi=0,
\end{eqnarray*}
where the prime ($'$) denotes derivative with respect to $u$. We write again the above parenthesis in terms of the angle $2\phi$ and we insert the value of $\Sigma$, obtaining
\begin{eqnarray}\label{eq:3.16}
 &&\left(6\tau\phi^2_uX^{1/2}-4\frac{A'}{A}\phi_uX+2\tau X^{3/2}\right)\sin2\phi\\
\nonumber && +\left(\frac{15}{2}\phi^2_uX-\frac{3}{2}\tau^2\phi^2_u+\left(2\tau\frac{A'}{A}-\frac{\tau'}{2}\right)\phi_uX^{1/2}+\frac{3}{2}X^2-\frac{\Sigma X}{2}\right)\cos2\phi\\
\nonumber &&+\left(\frac{3}{2}\tau^2\phi^2_u+\frac{9}{2}\phi^2_uX-\Big(2\tau\frac{A'}{A}-\frac{\tau'}{2}\Big)\phi_uX^{1/2}+\frac{5}{2}X^2+\frac{\Sigma X}{2}\right)=0.
\end{eqnarray}
Equation (\ref{eq:3.16}) is of the form
\begin{equation}\label{eq:3.17}
  P\sin2\phi+Q\cos2\phi+R=0,
\end{equation}
where
\begin{equation}\label{eq:3.18}
  P=\biggl(2\tau A^2+4\tau\phi^2_u\biggl)X^{1/2}+\biggl(-4AA'\phi_u+\frac{4A'}{A}\phi^3_u\biggl),
\end{equation}
\begin{equation}\label{eq:3.19}
  Q=\bigg(2\tau\frac{A'}{A}-\frac{\tau'}{2}\bigg)\phi_uX^{1/2}+\biggl(\frac{3}{2}A^4-\frac{\Sigma A^2}{2}+\Big(\frac{9A^2}{2}-\frac{3\tau^2}{2}+\frac{\Sigma}{2}\Big)\phi^2_u-6\phi^4_u \biggl),
\end{equation}
\begin{equation}\label{eq:3.20}
  R=-\bigg(2\tau\frac{A'}{A}-\frac{\tau'}{2}\bigg)\phi_uX^{1/2}+\biggl(\frac{5}{2}A^4+\frac{\Sigma A^2}{2}+\Big(\frac{3\tau^2}{2}-\frac{A^2}{2}-\frac{\Sigma}{2}\Big)\phi^2_u-2\phi^4_u \biggl),
\end{equation}
because of $X=A^2-\phi^2_u$ and $X^{3/2}=X^{1/2}X$.

We differentiate (\ref{eq:3.17}) with respect to $u$, obtaining 
\begin{equation}\label{eq:3.24}
  (P'-2Q\phi_u)\sin2\phi+(Q'+2P\phi_u)\cos2\phi+R'=0.
\end{equation}
The idea now is to  write this expression  as a linear combination of the trigonometric functions of type $\cos(n\phi)$ and $\sin(n\phi)$, which are independent linearly. We compute the three coefficients in (\ref{eq:3.24}). We calculate the derivatives $P'=\frac{\partial P}{\partial u}, Q'=\frac{\partial Q}{\partial u}$ and $R'=\frac{\partial R}{\partial u}$ by using (\ref{eq:2.5}) and (\ref{eq:3.4}). Then we obtain
\begin{eqnarray}\label{eq:3.21}
 P' &=& \bigg(\Big(6\tau A^2\phi_u-12\tau\phi^3_u\Big)X^{1/2}\\
\nonumber && +\Big(-4A^3A'+16AA'\phi^2_u-12\frac{A'}{A}\phi^4_u\Big)\bigg)\frac{\cos\phi}{\sin\phi}\\
\nonumber && +\left( 2\tau AA'+2A^2\tau'+\Big(4\tau'+24\tau\frac{A'}{A}\Big)\phi^2_u\right)X^{1/2} \\
\nonumber && +\left(6\tau^2A^2-12(A')^2-4A^2(\frac{A'}{A})'\right)\phi_u\\
\nonumber && +\left(-12\tau^2+12(\frac{A'}{A})^2+4(\frac{A'}{A})'\right)\phi^3_u,
\end{eqnarray}
\begin{eqnarray}\label{eq:3.22}
 Q' &=& \bigg(\Big( 2\tau AA'-\frac{\tau'A^2}{2}-\big(4\tau\frac{A'}{A}-\tau'\big)\phi^2_u\Big)X^{1/2}\\
\nonumber && +\big(9A^4-3\tau^2A^2+\Sigma A^2\big)\phi_u-\big(33A^2-3\tau^2+\Sigma\big)\phi^3_u\\
\nonumber && +24\phi^5_u \bigg)\frac{\cos\phi}{\sin\phi}+\Big(\big(\frac{A'\tau'}{A}+2\tau\big(\frac{A'}{A}\big)'-\frac{\tau''}{2}+4\tau\big(\frac{A'}{A}\big)^2\\
\nonumber && +9A^2\tau-3\tau^3+\Sigma\tau\big)\phi_u-24\tau\phi^3_u\Big)X^{1/2} \\
\nonumber && + \big(6A^3A'-\frac{\Sigma' A^2}{2}-\Sigma AA'+2AA'\tau^2-\frac{A^2\tau\tau'}{2}\big)\\
\nonumber && +\big(18AA'-2\tau\tau'-\frac{7A'\tau^2}{A}+\frac{A'\Sigma}{A}+\frac{\Sigma'}{2}\big)\phi^2_u-24\frac{A'}{A}\phi^4_u,
\end{eqnarray}
\begin{eqnarray}\label{eq:3.23}
 R' &=& \bigg(\Big(-2\tau AA'+\frac{\tau'A^2}{2}+\big(4\tau\frac{A'}{A}-\tau'\big)\phi^2_u\Big)X^{1/2}\\
\nonumber && +\big(3\tau^2A^2-A^4-\Sigma A^2\big)\phi_u  \\
\nonumber && -\big(7A^2+3\tau^2-\Sigma\big)\phi^3_u+8\phi^5_u\bigg)\frac{\cos\phi}{\sin\phi} \\
\nonumber && +\Big(\big(2\tau^3-6\tau(\frac{A'}{A})^2-\tau(\frac{A'}{A})'-\frac{A'\tau'}{A}-A^2\tau+\frac{\tau''}{2}\big)\phi_u\\
\nonumber && -8\tau\phi_u^3\Big)X^{1/2}+\Big(10A^3A'+\frac{A^2}{2}(\tau\tau'+\Sigma')+(\Sigma-2\tau^2)AA'\Big)\\
\nonumber && +\Big(2\tau\tau'-2AA'-\frac{\Sigma'}{2}+\frac{7A'\tau^2}{A}-\frac{A'\Sigma}{A}\Big)\phi_u^2-\frac{8A'}{A}\phi^4_u.
\end{eqnarray}
Now we compute the  coefficients in (\ref{eq:3.24}). From Eqs. (\ref{eq:3.18})--(\ref{eq:3.23}), we get the coefficient $P'-2Q\phi_u$ of $\sin2\phi$, obtaining  
\begin{equation}\label{eq:3.25}
  P'-2Q\phi_u=\Big(\alpha_1X^{1/2}+\alpha_2\Big)\frac{\cos\phi}{\sin\phi}+(\alpha_3X^{1/2}+\alpha_4),
\end{equation}
where
\begin{eqnarray*}
  \alpha_1 &=& 6\tau A^2\phi_u-12\tau\phi^3_u\\
  \alpha_2 &=& -4A^3A'+16AA'\phi^2_u-12\frac{A'}{A}\phi^4_u \\
  \alpha_3 &=& 2\tau AA'+2A^2\tau'+\Big(20\tau\frac{A'}{A}+5\tau'\Big)\phi^2_u \\
  \alpha_4 &=& \bigg(6\tau^2A^2-12(A')^2-4A^2(\frac{A'}{A})'-3A^4+\Sigma A^2\bigg)\phi_u\\
  &&   +\bigg(12(\frac{A'}{A})^2+4(\frac{A'}{A})'-9A^2-9\tau^2-\Sigma\bigg)\phi^3_u+12\phi^5_u,
\end{eqnarray*}
On the other hand, the coefficient $Q'+2P\phi_u$ of $\cos2\phi$ is
\begin{equation}\label{eq:3.26}
  Q'+2P\phi_u=\Big(\beta_1X^{1/2}+\beta_2\Big)\frac{\cos\phi}{\sin\phi}+(\beta_3X^{1/2}+\beta_4),
\end{equation}
where
\begin{eqnarray*}
  \beta_1 &=& \Big(2\tau AA'-\frac{\tau'}{2}A^2\Big)-\Big(4\tau\frac{A'}{A}-\tau'\Big)\phi^2_u \\
  \beta_2 &=& (9A^4-3A^2\tau^2+\Sigma A^2)\phi_u-(33A^2-3\tau^2+\Sigma)\phi^3_u+24\phi^5_u\\
  \beta_3 &=& \Big(\frac{A'\tau'}{A}+2\tau(\frac{A'}{A})'-\frac{\tau''}{2}+4\tau(\frac{A'}{A})^2+13A^2\tau-3\tau^3+\Sigma\tau\Big)\phi_u-16\tau\phi^3_u \\
  \beta_4 &=& \Big(6A^3A'-\frac{\Sigma'A^2}{2}-\Sigma AA'+2AA'\tau^2-\frac{A^2\tau\tau'}{2}\Big)\\
  &&   +\Big(10AA'-2\tau\tau'-\frac{7A'}{A}\tau^2+\frac{A'\Sigma}{A}+\frac{\Sigma'}{2}\Big)\phi^2_u -\frac{16A'}{A}\phi^4_u.
\end{eqnarray*}
Finally, the term $R'$ is
\begin{equation}\label{eq:3.27}
  R'=\Big(\gamma_1X^{1/2}+\gamma_2\Big)\frac{\cos\phi}{\sin\phi}+(\gamma_3X^{1/2}+\gamma_4),
\end{equation}
where
\begin{eqnarray*}
  \gamma_1 &=& -2\tau AA'+\frac{\tau'}{2}A^2 +\Big(4\tau\frac{A'}{A}-\tau'\Big)\phi^2_u \\
  \gamma_2 &=& \left(3A^2\tau^2-A^4-\Sigma A^2\right)\phi_u-\left(7A^2+3\tau^2-\Sigma\right)\phi^3_u+8\phi^5_u\\
  \gamma_3 &=& \left(2\tau^3-6\tau(\frac{A'}{A})^2-\tau(\frac{A'}{A})'-\frac{\tau'A'}{A}-A^2\tau+\frac{\tau''}{2}\right)\phi_u-8\tau\phi^3_u \\
  \gamma_4 &=&  \left(10A^3A'+\frac{A^2}{2}(\tau\tau'+\Sigma')+(\Sigma-2\tau^2)AA'\right)\\
  &&   +\left(2\tau\tau'-2AA'-\frac{\Sigma'}{2}+\frac{7A'}{A}\tau^2-\frac{A'}{A}\Sigma\right)\phi^2_u-\frac{8A'}{A}\phi^4_u.
\end{eqnarray*}
Putting all these in (\ref{eq:3.24}) we obtain
\begin{eqnarray*}
&& \left((\alpha_1X^{1/2}+\alpha_2)\frac{\cos\phi}{\sin\phi}+(\alpha_3X^{1/2}+\alpha_4)\right)\sin2\phi\\
&& +\left((\beta_1X^{1/2}+\beta_2)\frac{\cos\phi}{\sin\phi}+(\beta_3X^{1/2}+\beta_4)\right)\cos2\phi\\
&& +\left((\gamma_1X^{1/2}+\gamma_2)\frac{\cos\phi}{\sin\phi}+(\gamma_3X^{1/2}+\gamma_4)\right)=0.
\end{eqnarray*}
Using $ \cos\phi/\sin\phi= \sin2\phi/(1-\cos2\phi)$, the above expression can be written as
\begin{eqnarray*}
&& -\left((\alpha_1X^{1/2}+\alpha_2)+(\beta_3X^{1/2}+\beta_4)\right)\cos^22\phi\\
&& +\left((\beta_1X^{1/2}+\beta_2)-(\alpha_3X^{1/2}+\alpha_4)\right)\sin2\phi \cos2\phi\\
&& +\left((\beta_3X^{1/2}+\beta_4)-(\gamma_3X^{1/2}+\gamma_4)\right)\cos2\phi\\
&& +\left((\alpha_3X^{1/2}+\alpha_4)+(\gamma_1X^{1/2}+\gamma_2)\right)\sin2\phi\\
&& +\left((\alpha_1X^{1/2}+\alpha_2)+(\gamma_3X^{1/2}+\gamma_4)\right)=0.
\end{eqnarray*}
Multiplying the last equation by $P$ and inserting $P\sin2\phi=-Q\cos2\phi-R$  from Eq.  (\ref{eq:3.17}), we have
\begin{eqnarray}\label{eq:3.28}
 && \biggl(Q(\alpha_3X^{1/2}+\alpha_4)-Q(\beta_1X^{1/2}+\beta_2)-P(\alpha_1X^{1/2}+\alpha_2) \\
\nonumber && -P(\beta_3X^{1/2}+\beta_4)\biggl)\cos^22\phi+\biggl(P(\beta_3X^{1/2}+\beta_4)-P(\gamma_3X^{1/2}+\gamma_4) \\
\nonumber && -Q(\alpha_3X^{1/2}+\alpha_4)-Q(\gamma_1X^{1/2}+\gamma_2)+R(\alpha_3X^{1/2}+\alpha_4)\\
\nonumber && -R(\beta_1X^{1/2}+\beta_2)\biggl)\cos2\phi+\biggl(P(\alpha_1X^{1/2}+\alpha_2)+P(\gamma_3X^{1/2}+\gamma_4)\\
\nonumber && -R(\alpha_3X^{1/2}+\alpha_4)-R(\gamma_1X^{1/2}+\gamma_2)\biggl)=0.
\end{eqnarray}

Equation (\ref{eq:3.28}) is of the form
\begin{equation}\label{eq:3.33}
  b_2\cos^22\phi+b_1\cos2\phi+b_0=0,
\end{equation}
where
\begin{equation}\label{eq:3.34}
  b_2=b_{21}X^{1/2}+b_{22},\quad    b_1=b_{11}X^{1/2}+b_{12}\quad 
   b_0=b_{01}X^{1/2}+b_{02},
\end{equation}
and
\begin{eqnarray}\label{eq:3.35}
   b_{21} &=& Q_2(\alpha_3-\beta_1)+Q_1(\alpha_4-\beta_2)-P_1(\alpha_2+\beta_4)-P_2(\alpha_1+\beta_3) \\
\nonumber   b_{22} &=& Q_1(\alpha_3-\beta_1)X-P_1(\alpha_1+\beta_3)X+Q_2(\alpha_4-\beta_2)-P_2(\alpha_2+\beta_4) \\
\nonumber   b_{11} &=& P_1(\beta_4-\gamma_4)+P_2(\beta_3-\gamma_3)-Q_1(\alpha_4+\gamma_2) \\
\nonumber && -Q_2(\alpha_3+\gamma_1)+R_1(\alpha_4-\beta_2)+R_2(\alpha_3-\beta_1) \\
\nonumber   b_{12} &=& P_1(\beta_3-\gamma_3)X+P_2(\beta_4-\gamma_4)-Q_1(\alpha_3+\gamma_1)X \\
\nonumber && -Q_2(\alpha_4+\gamma_2)+R_1(\alpha_3-\beta_1)X+R_2(\alpha_4-\beta_2) \\
\nonumber   b_{01} &=& P_1(\alpha_2+\gamma_4)+P_2(\alpha_1+\gamma_3)-R_1(\alpha_4+\gamma_2)-R_2(\alpha_3+\gamma_1) \\
\nonumber   b_{02} &=& P_1(\alpha_1+\gamma_3)X+P_2(\alpha_2+\gamma_4)-R_1(\alpha_3+\gamma_1)X-R_2(\alpha_4+\gamma_2).
\end{eqnarray}

On the other hand Eq. (\ref{eq:3.17}) gives
\begin{equation}\label{eq:3.29}
  (P^2+Q^2)\cos^22\phi+2QR\cos2\phi+R^2-P^2=0.
\end{equation}
Equations (\ref{eq:3.18}), (\ref{eq:3.19}) and (\ref{eq:3.20}) are written  as 
$$P = P_1X^{1/2}+P_2,\quad Q= Q_1X^{1/2}+Q_2,\quad R = R_1X^{1/2}+R_2,$$
 respectively, where 
\begin{equation}\label{eq:3.30}
  P_1 = 2\tau A^2+4\tau\phi^2_u,\quad   P_2 = -4AA'\phi_u+\frac{4A'}{A}\phi^3_u,
\end{equation}
\begin{eqnarray}\label{eq:3.31}
  Q_1 &=& \bigg(2\tau\frac{A'}{A}-\frac{\tau'}{2}\bigg)\phi_u, \\
\nonumber  Q_2 &=& \frac{3}{2}A^4-\frac{\Sigma A^2}{2}+\bigg(\frac{9A^2}{2}-\frac{3\tau^2}{2}+\frac{\Sigma}{2}\bigg)\phi^2_u-6\phi^4_u,
\end{eqnarray}
\begin{eqnarray}\label{eq:3.32}
 R_1 &=& -\bigg(2\tau\frac{A'}{A}-\frac{\tau'}{2}\bigg)\phi_u, \\
\nonumber  R_2 &=& \frac{5}{2}A^4+\frac{\Sigma A^2}{2}+\bigg(-\frac{A^2}{2}+\frac{3\tau^2}{2}-\frac{\Sigma}{2}\bigg)\phi^2_u-2\phi^4_u.
\end{eqnarray}
Thus Eq.  (\ref{eq:3.29}) is written as
\begin{equation}\label{eq:3.36}
  c_2\cos^22\phi+c_2\cos2\phi+c_0=0,
\end{equation}
where
\begin{equation}\label{eq:3.37}
   c_2= c_{21}X^{1/2}+c_{22},\quad 
  c_1= c_{11}X^{1/2}+c_{12},\quad
   c_0= c_{01}X^{1/2}+c_{02},
\end{equation}
and
\begin{eqnarray}\label{eq:3.38}
  c_{21} &=& 2P_1P_2+2Q_1Q_2 \\
\nonumber  c_{22} &=& P_1^2X+Q_1^2X+P_2^2+Q_2^2 \\
\nonumber  c_{11} &=& 2(Q_1R_2+Q_2R_1) \\
\nonumber  c_{12} &=& 2(Q_1R_1X+Q_2R_2) \\
\nonumber  c_{01} &=& 2(R_1R_2-P_1P_2) \\
\nonumber  c_{02} &=& R_1^2X-P_1^2X+R_2^2-P_2^2.
\end{eqnarray}
The quadratic equations (\ref{eq:3.33}) and (\ref{eq:3.36}) have at least a common solution with respect to $\cos2\phi$. Thus, their eliminant must be zero, hence, we have
\begin{equation}\label{eq:3.39}
  (b_2c_0-b_0c_2)^2-(b_0c_1-b_1c_0)(b_1c_2-b_2c_1)=0.
\end{equation}
The proof of Th. \ref{t1} finishes by observing that this identity can be written as a polynomial equation on $\phi_u$, and as in the proof of Prop. \ref{proposition:3.2}, all coefficients of this polynomial must vanish. 

Taking  into account (\ref{eq:3.34}) and (\ref{eq:3.37}), we compute the  three factors of (\ref{eq:3.39}), obtaining 
\begin{eqnarray}\label{eq:3.40}
  (b_2c_0-b_0c_2)&=&\Big(b_{22}c_{01}+b_{21}c_{02}-b_{01}c_{22}-b_{02}c_{21}\Big)X^{1/2}\\
\nonumber  &&+\Big(b_{21}c_{01}X+b_{22}c_{02}-b_{01}c_{21}X-b_{02}c_{22}\Big)\\
\nonumber  &=&\kappa_1X^{1/2}+\kappa_2,
\end{eqnarray}
\begin{eqnarray}\label{eq:3.41}
  (b_0c_1-b_1c_0)&=&\Big(b_{01}c_{12}+b_{02}c_{11}-b_{11}c_{02}-b_{12}c_{01}\Big)X^{1/2}\\
\nonumber  && + \Big(b_{01}c_{11}X+b_{02}c_{12}-b_{11}c_{01}X-b_{12}c_{02}\Big)\\
\nonumber  &=&\lambda_1X^{1/2}+\lambda_2,
\end{eqnarray}
\begin{eqnarray}\label{eq:3.42}
(b_1c_2-b_2c_1)&=&\Big(b_{11}c_{22}+b_{12}c_{21}-b_{21}c_{12}-b_{22}c_{11}\Big)X^{1/2}\\
\nonumber   &&+\Big(b_{11}c_{21}X+b_{12}c_{22}-b_{21}c_{11}X-b_{22}c_{12}\Big)\\
\nonumber  &=&\mu_1X^{1/2}+\mu_2.
\end{eqnarray}
 Putting all these in (\ref{eq:3.39}) we obtain 
  $$   (\kappa_1^{2}X+\kappa_2^2-\lambda_1\mu_1X-\lambda_2\mu_2)=(\lambda_1\mu_2+\lambda_2\mu_1-2\kappa_1\kappa_2)X^{1/2}, $$ 
or raising squared
 \begin{equation}\label{eq:3.43}
  (\kappa_1^{2}X+\kappa_2^2-\lambda_1\mu_1X-\lambda_2\mu_2)^2-(\lambda_1\mu_2+\lambda_2\mu_1-2\kappa_1\kappa_2)^2X=0.
\end{equation}
The last equation is a polynomial equation with respect to $\phi_u$, with coefficients which depend only on $u$. Thus, all coefficients must be zero as, in the proof of Prop. \ref{proposition:3.2}, we claimed.
In order to calculate the coefficients of this equation, at least those ones of higher exponent, we proceed as follows. We set for simplification $z=\phi_u$ and we express the terms $\alpha_i$, $\beta_i$ and $\gamma_i$ as polynomials on $z$. By using (\ref{eq:3.25}), (\ref{eq:3.26}) and (\ref{eq:3.27}), we have
 \begin{eqnarray}\label{eq:3.44}
 \alpha_1&=& \alpha_{11}z+\alpha_{13}z^3\\
\nonumber \alpha_2&=& \alpha_{20}+\alpha_{22}z^2+\alpha_{24}z^4\\
\nonumber \alpha_3&=& \alpha_{30}+\alpha_{32}z^2\\
\nonumber \alpha_4&=& \alpha_{41}z+\alpha_{43}z^3+\alpha_{45}z^5
\end{eqnarray}
 with
 \begin{eqnarray*}
 &&\alpha_{11}=6\tau A^2,  \alpha_{13}=-12\tau,  \alpha_{20}=-4A^3A', \alpha_{22}=16AA', \\
 && \alpha_{24}=-12\frac{A'}{A}, \alpha_{30}=2\tau AA'+2A^2\tau, \alpha_{32}=20\tau\frac{A'}{A}+5\tau',\\
 &&\alpha_{41}=6\tau^2A^2-12(A')^2-4A^2(\frac{A'}{A})'-3A^4+\Sigma A^2,\\
 &&\alpha_{43}=12(\frac{A'}{A})^2+4(\frac{A'}{A})'-9A^2-9\tau^2-\Sigma, \alpha_{45}=12.
 \end{eqnarray*}
 For $\beta_i$, we have
 \begin{eqnarray}\label{eq:3.45}
 \beta_1&=& \beta_{10}+\beta_{12}z^2\\
\nonumber \beta_2&=& \beta_{21}z+\beta_{23}z^3+\beta_{25}z^5\\
\nonumber \beta_3&=& \beta_{31}z+\beta_{33}z^3\\
\nonumber \beta_4&=& \beta_{40}+\beta_{42}z^2+\beta_{44}z^4
\end{eqnarray}
 with
 \begin{eqnarray*}
 &&\beta_{10}=2\tau AA'-\frac{\tau'}{2}A^2, \beta_{12}=4\tau\frac{A'}{A}-\tau',\\
 && \beta_{21}=9A^4-3A^2\tau^2+\Sigma A^2, \beta_{23}=-(33A^2-3\tau^2+\Sigma), \beta_{25}=24, \\
 &&\beta_{31}=\frac{A'\tau'}{A}+2\tau(\frac{A'}{A})'-\frac{\tau''}{2}+4\tau(\frac{A'}{A})^2+13A^2\tau-3\tau^3+\Sigma\tau, \\
 &&\beta_{33}=16\tau, \beta_{40}=6A^3A'-\frac{\Sigma'A^2}{2}-\Sigma AA'+2AA'\tau^2-\frac{A^2\tau\tau'}{2}, \\
 &&\beta_{42}=10AA'-2\tau\tau'-\frac{7A'}{A}\tau^2+\frac{A'\Sigma}{A}+\frac{\Sigma'}{2},  \beta_{44}= -\frac{16A'}{A}.
 \end{eqnarray*}
 And for $\gamma_i$, we  have
 \begin{eqnarray}\label{eq:3.46}
 \gamma_1&=& \gamma_{10}+\gamma_{12}z^2\\
\nonumber \gamma_2&=& \gamma_{21}z+\gamma_{23}z^3+\gamma_{25}z^5\\
\nonumber \gamma_3&=& \gamma_{31}z+\gamma_{33}z^3\\
\nonumber \gamma_4&=& \gamma_{40}+\gamma_{42}z^2+\gamma_{44}z^4,
\end{eqnarray}
 with
\begin{eqnarray*}
&&\gamma_{10}=-2\tau AA'+\frac{\tau'}{2}A^2, \gamma_{12}=4\tau\frac{A'}{A}-\tau', \\
&& \gamma_{21}=3A^2\tau^2-A^4-\Sigma A^2, \gamma_{23}=7A^2+3\tau^2-\Sigma, \gamma_{25}=8,\\
&& \gamma_{31}=2\tau^3-6\tau(\frac{A'}{A})^2-\tau(\frac{A'}{A})'-\frac{\tau'A'}{A}-A^2\tau+\frac{\tau''}{2}, \\
&&\gamma_{33}=-8\tau, \gamma_{40}=10A^3A'+\frac{A^2}{2}(\tau\tau'+\Sigma')+(\Sigma-2\tau^2)AA', \\
&&\gamma_{42}=2\tau\tau'-2AA'-\frac{\Sigma'}{2}+\frac{7A'}{A}\tau^2-\frac{A'}{A}\Sigma, \gamma_{44}=-\frac{8A'}{A}.
\end{eqnarray*}
We return to Eq. (\ref{eq:3.29}) by computing the coefficients $P$, $Q$ and $R$ as  polynomials on $z$. From (\ref{eq:3.30}), (\ref{eq:3.31}) and (\ref{eq:3.32}) we  have
\begin{equation}\label{eq:3.47}
 P_1=P_{10}+P_{12}z^2,\quad P_2=P_{21}z+P_{23}z^3,
\end{equation}
with
$P_{10}=2\tau A^2, P_{12}=4\tau, P_{21}=-4AA', P_{23}=4A'/A$,
\begin{equation}\label{eq:3.48}
          Q_1=Q_{11}z,\quad  Q_2=Q_{20}z+Q_{22}z^2+Q_{24}z^4,
\end{equation}
with
$Q_{11}=2\tau\frac{A'}{A}-\frac{\tau'}{2}, Q_{20}=\frac{3}{2}A^4-\frac{\Sigma A^2}{2}, Q_{22}=\frac{9A^2}{2}-\frac{3\tau^2}{2}+\frac{\Sigma}{2}, Q_{24}=-6$, and
\begin{equation}\label{eq:3.49}
          R_1=R_{11}z,\quad R_2=R_{20}z+R_{22}z^2+R_{24}z^4,
\end{equation}
with
$$R_{11}=-\bigg(2\tau\frac{A'}{A}-\frac{\tau'}{2}\bigg), R_{20}=\frac{5}{2}A^4+\frac{\Sigma A^2}{2}, R_{22}=-\frac{A^2}{2}+\frac{3\tau^2}{2}-\frac{\Sigma}{2}, R_{24}=-2.$$
Now we compute $b_{ij}$ and $c_{ij}$ of (\ref{eq:3.35}) and (\ref{eq:3.38}) in terms of the variable $z$. 
\subsection{Computation of $b_{ij}$}

\begin{enumerate}[label=(\roman*)]
\item Computation of $b_{21}$. From (\ref{eq:3.35}) we get
\begin{eqnarray*}
  b_{21} &=& Q_2(\alpha_3-\beta_1)+Q_1(\alpha_4-\beta_2)-P_1(\alpha_2+\beta_4)-P_2(\alpha_1+\beta_3) \\
   &=& \Big(Q_{20}+Q_{22}z^2+Q_{24}z^4\Big)\Big((\alpha_{30}-\beta_{10})+(\alpha_{32}-\beta_{12})z^2\Big) \\
   && +Q_{11}z\Big((\alpha_{41}-\beta_{21})z+(\alpha_{43}-\beta_{23})z^3)+(\alpha_{45}-\beta_{25})z^5\Big) \\
   && -(P_{10}+P_{12}z^2)\Big((\alpha_{20}+\beta_{40})+(\alpha_{22}+\beta_{42})z^2+(\alpha_{24}+\beta_{44})z^4\Big)  \\
   && -(P_{21}z+P_{23}z^3)\Big((\alpha_{11}+\beta_{31})z+(\alpha_{13}+\beta_{33})z^3\Big) \\
   &=& Q_{20}(\alpha_{30}-\beta_{10})-P_{10}(\alpha_{20}+\beta_{40}) \\
   &&+\Big(Q_{20}(\alpha_{32}-\beta_{12})+Q_{22}(\alpha_{30}-\beta_{10})+Q_{11}(\alpha_{41}-\beta_{21})-P_{10}(\alpha_{22}+\beta_{42})\\
   &&-P_{12}(\alpha_{20}+\beta_{40})-P_{21}(\alpha_{11}+\beta_{31})\Big)z^2 \\
   &&+\Big(Q_{22}(\alpha_{32}-\beta_{12})+Q_{24}(\alpha_{30}-\beta_{10})+Q_{11}(\alpha_{43}-\beta_{23})-P_{10}(\alpha_{24}+\beta_{44})\\
   &&-P_{12}(\alpha_{22}+\beta_{42})-P_{21}(\alpha_{13}+\beta_{33})-P_{23}(\alpha_{11}+\beta_{31})\Big)z^4 \\
   &&+\Big(Q_{24}(\alpha_{32}-\beta_{12})+Q_{11}(\alpha_{45}-\beta_{25})-P_{12}(\alpha_{24}+\beta_{44})-P_{23}(\alpha_{13}+\beta_{33})\Big)z^6\\
  &=&b_{210}+b_{212}z^2+b_{214}z^4+b_{216}z^6,
\end{eqnarray*}
where
\begin{eqnarray*}
  b_{216}&=&Q_{24}(\alpha_{32}-\beta_{12})+Q_{11}(\alpha_{45}-\beta_{25})-P_{12}(\alpha_{24}+\beta_{44})-P_{23}(\alpha_{13}+\beta_{33})\nonumber \\
  &=& -6\bigg(20\tau\frac{A'}{A}+5\tau'+4\tau\frac{A'}{A}-\tau'\bigg)+\bigg(2\tau\frac{A'}{A}-\frac{\tau'}{2}\bigg)(12-24) \nonumber\\
  && -4\tau\bigg(-12\frac{A'}{A}-16\frac{A'}{A}\bigg)-4\frac{A'}{A}(-12\tau-16\tau)\\
  &=& 56\tau\frac{A'}{A}-18\tau'.\label{eq:3.50}
\end{eqnarray*}
\item Computation of $b_{22}$. In a similar way, we find
$$b_{22}=b_{221}z+b_{223}z^3+b_{225}z^5+b_{227}z^7+b_{229}z^9.$$
where
$$  b_{229}=72,\qquad 
  b_{227}=-198A^2-28\tau^2-18(\frac{A'}{A})'+28(\frac{A'}{A})^2.$$
  \item Computation of $b_{11}$. We have
$$b_{11}=b_{110}+b_{112}z^2+b_{114}z^4+b_{116}z^6,$$
where
$$  b_{116}=16\tau\frac{A'}{A}+20\tau'.$$
\item Computation of $b_{12}$. Now
$$b_{12}=b_{121}z+b_{123}z^3+b_{125}z^5+b_{127}z^7+b_{129}z^9,$$
where
$$  b_{129}=144,\qquad 
  b_{127}=-228A^2-10\tau^2+20(\frac{A'}{A})'+8(\frac{A'}{A})^2.$$
  \item Computation of $b_{01}$. Now 
$$b_{01}=b_{010}+b_{012}z^2+b_{014}z^4+b_{016}z^6,$$
where
$$  b_{016}=-72\tau\frac{A'}{A}-2\tau'.$$
\item Computation of $b_{02}$. Finally
$$b_{02}=b_{021}z+b_{023}z^3+b_{025}z^5+b_{027}z^7+b_{029}z^9,$$
where
$$  b_{029}=40,\qquad  b_{027}=-2(\frac{A'}{A})'-36(\frac{A'}{A})^2+36\tau^2-22A^2.$$\end{enumerate}
\subsection{Computation of $c_{ij}$}
\begin{enumerate}[label=(\roman*)]
\item Computation of $c_{21}$. From (\ref{eq:3.38}) we  find
\begin{eqnarray*}
  c_{21} &=& 2P_1P_2+2Q_1Q_2=2\Big(P_{10}+P_{12}z^2\Big)\Big(P_{21}z+P_{23}z^3\Big) \\
   && +2Q_{11}z\Big(Q_{20}+Q_{22}z^2+Q_{24}z^4\Big) \\
   &=& \Big(2P_{10}P_{21}+2Q_{11}Q_{20}\Big)z+\Big(2P_{10}P_{23}+2P_{12}P_{21}+2Q_{11}Q_{22}\Big)z^3\\
   &&  +\Big(2P_{12}P_{23}+2Q_{11}Q_{24}\Big)z^5\\
   &=&c_{211}z+c_{213}z^3+c_{215}z^5,
\end{eqnarray*}
where
\begin{eqnarray*}
  c_{215}&=&2P_{12}P_{23}+2Q_{11}Q_{24}=2\cdot4\tau\cdot\frac{4A'}{A}+2\bigg(\frac{2\tau A'}{A}-\frac{\tau'}{2}\bigg)(-6)\\
  &=&\frac{8\tau A'}{A}+6\tau'. 
\end{eqnarray*}
\item Computation of $c_{22}$. In a similar way we obtain
$$c_{22}=c_{220}+c_{222}z^2+c_{224}z^4+c_{226}z^6+c_{228}z^8,$$
$$c_{228}=36,\qquad
  c_{226}=6(\frac{A'}{A})'+4(\frac{A'}{A})^2-4\tau^2-54A^2.$$
\item Computation of $c_{11}$. We have $c_{11}=c_{111}z+c_{113}z^3+c_{115}z^5$, 
where
$$   c_{115}=16\tau\frac{A'}{A}-4\tau',$$
\item Computation of $c_{12}$. We have 
$$c_{12}=c_{120}z+c_{122}z^2+c_{124}z^4+c_{126}z^6+c_{128}z^8,$$
where
$$  c_{128}=24,\qquad  
  c_{126}=-4(\frac{A'}{A})'+8(\frac{A'}{A})^2-8\tau^2-12A^2.$$
  \item Computation of $c_{01}$. For $c_{01}$ we have $c_{01}=c_{011}z+c_{013}z^3+c_{015}z^5$, 
with
$$   c_{015}=-24\tau\frac{A'}{A}-2\tau'.$$
\item Computation of $c_{02}$. Finally, $c_{02}=c_{020}z+c_{022}z^2+c_{024}z^4+c_{026}z^6+c_{028}z^8$, 
where
$$  c_{028}=4,\qquad 
  c_{026}=-2(\frac{A'}{A})'-12(\frac{A'}{A})^2+12\tau^2+2A^2.$$
  \end{enumerate}

Once obtained $b_{ij}$ and $c_{ij}$, we utilize them in (\ref{eq:3.40}), (\ref{eq:3.41}) and (\ref{eq:3.42}) to get the values $\kappa_i$, $\lambda_i$ and $\mu_i$ and finally, put all this information in  (\ref{eq:3.43}).  In what follows,  we only indicate the powers of $z$ of higher exponents. 
\begin{enumerate}
\item Computation of $\kappa_1$ and $\kappa_2$. From (\ref{eq:3.40}) we have
\begin{eqnarray*}
  \kappa_1 &=& b_{22}c_{01}+b_{21}c_{02}-b_{01}c_{22}-b_{02}c_{21} \\
   &=& \Big(b_{221}z+b_{223}z^3+b_{225}z^5+b_{227}z^7+b_{229}z^9\Big)\Big(c_{011}z+c_{013}z^3+c_{015}z^5\Big)  \\
   && +\Big(b_{210}+b_{212}z^2+b_{214}z^4+b_{216}z^6\Big)\Big(c_{020}+c_{022}z^2+c_{024}z^4+c_{026}z^6\\
   && +c_{028}z^8\Big)-\Big(b_{010}+b_{012}z^2+b_{014}z^4+b_{016}z^6\Big)\Big(c_{220}+c_{222}z^2+c_{224}z^4\\
   && +c_{226}z^6+c_{228}z^8\Big)-\Big(b_{021}z+b_{023}z^3+b_{025}z^5+b_{027}z^7+b_{029}z^9\Big)\Big(c_{211}z\\
   && +c_{213}z^3+c_{215}z^5\Big)\\
   &=&\ldots+\kappa_{110}z^{10}+\kappa_{112}z^{12}+\kappa_{114}z^{14},
\end{eqnarray*}where
$$\kappa_{114}=b_{229}c_{015}+b_{216}c_{028}-b_{016}c_{228}-b_{029}c_{215}.$$
Substituting in the last equation the preceding calculations we find
\begin{equation}\label{eq:3.68}
  \kappa_{114}= 768\tau\frac{A'}{A}-384\tau'.
\end{equation}
In a similar way we obtain $\kappa_{2}=\ldots+\kappa_{213}z^{13}+\kappa_{215}z^{15}+\kappa_{217}z^{17}$ where
\begin{eqnarray}\label{eq:3.69}
  \kappa_{215}&=&b_{227}c_{028}+b_{229}c_{026}-b_{027}c_{228}-b_{029}c_{226}\\
\nonumber  &=&-384(\frac{A'}{A})'+384(\frac{A'}{A})^2+2304A^2-384\tau^2\\
   \kappa_{217}&=&b_{229}c_{028}-b_{029}c_{228}=-1152.\label{eq:3.70}
\end{eqnarray}
\item Computation of $\lambda_1$ and $\lambda_2$. We have $\lambda_{1}=\ldots+\lambda_{110}z^{10}+\lambda_{112}z^{12}+\lambda_{114}z^{14}$,
with
\begin{equation}\label{eq:3.71}
  \lambda_{114}= b_{016}c_{128}+b_{029}c_{115}-b_{116}c_{028}-b_{129}c_{015}=2304\tau\frac{A'}{A}.
\end{equation}
Also we have $\lambda_{2}=\ldots+\lambda_{213}z^{13}+\lambda_{215}z^{15}+\lambda_{217}z^{17}$, with
\begin{eqnarray}\label{eq:3.72}
  \lambda_{215}&=& b_{027}c_{128}+b_{029}c_{126}-b_{127}c_{028}-b_{129}c_{026}\\
\nonumber  &=&1152(\frac{A'}{A})^2-1144\tau^2-384A^2\\
  \lambda_{217}&=& b_{029}c_{128}-b_{129}c_{028}=384.\label{eq:3.73}
\end{eqnarray}
\item Computation of $\mu_1$ and $\mu_2$. Now  $\mu_{1}=\ldots+\mu_{110}z^{10}+\mu_{112}z^{12}+\mu_{114}z^{14}$, with
\begin{eqnarray}\label{eq:3.74}
  \mu_{114}&=& b_{116}c_{228}+b_{129}c_{215}-b_{216}c_{128}-b_{229}c_{115}\\
\nonumber  &=&-768\tau\frac{A'}{A}+2304\tau'.
\end{eqnarray}
For $\mu_2$, we have   $\mu_{2}=\ldots+\mu_{213}z^{13}+\mu_{215}z^{15}+\mu_{217}z^{17}$, with
\begin{eqnarray}\label{eq:3.75}
  \mu_{215}&=& b_{127}c_{228}+b_{129}c_{226}-b_{227}c_{128}-b_{229}c_{126}\\
\nonumber  &=&2304(\frac{A'}{A})'-384(\frac{A'}{A})^2+312\tau^2-10368A^2\\
  \mu_{217}&= &b_{129}c_{228}-b_{229}c_{128}=3456.\label{eq:3.76}
\end{eqnarray}
\end{enumerate}
The three leading terms of Eq.  (\ref{eq:3.43}) are $z^{68}, z^{66}$ and $z^{64}$ with corresponding coefficients
$$z^{68}\rightarrow (\kappa^2_{217}-\lambda_{217}\mu_{217})^2,$$
$$ z^{66}\rightarrow 2(\kappa^2_{217}-\lambda_{217}\mu_{217})(2\kappa_{215}\kappa_{217}-\lambda_{215}\mu_{217}-\lambda_{217}\mu_{215}),$$
and
\begin{eqnarray*}
 z^{64}\rightarrow &&(2\kappa_{215}\kappa_{217}-\lambda_{215}\mu_{217}-\lambda_{217}\mu_{215})^2+(\lambda_{114}\mu_{217}+\lambda_{217}\mu_{114}-2\kappa_{114}\kappa_{217}) ^2 \\
 &&+2(\kappa^2_{217}-\lambda_{217}\mu_{217})(-\kappa^2_{114}+\lambda_{114}\mu_{114}+\kappa^2_{115}\\
 &&+2\kappa_{213}\kappa_{217}-\lambda_{213}\mu_{217}-\lambda_{217}\mu_{213}-\lambda_{215}\mu_{215}).
\end{eqnarray*}

Since $\kappa_{217}=-1152, \lambda_{217}=384$ and $\mu_{217}=3456$ we have $\kappa^2_{217}-\lambda_{217}\mu_{217}=0$, and thus the coefficients of $z^{66}$ and $z^{68}$ are zero trivially. However the coefficient of $z^{64}$ becomes
$$(2\kappa_{215}\kappa_{217}-\lambda_{215}\mu_{217}-\lambda_{217}\mu_{215})^2+(\lambda_{114}\mu_{217}+\lambda_{217}\mu_{114}-2\kappa_{114}\kappa_{217})^2=0.$$
So we deduce
\begin{equation}\label{eq:3.77}
2\kappa_{215}\kappa_{217}-\lambda_{215}\mu_{217}-\lambda_{217}\mu_{215}=0
\end{equation}
and
\begin{equation}\label{eq:3.78}
\lambda_{114}\mu_{217}+\lambda_{217}\mu_{114}-2\kappa_{114}\kappa_{217}=0.
\end{equation}
Taking into account (\ref{eq:3.68})-(\ref{eq:3.76}), Equation  (\ref{eq:3.77}) gives
\begin{eqnarray*}
\nonumber  &-&3^2\cdot2^8\Big(-384(\frac{A'}{A})'+384(\frac{A'}{A})^2+2304A^2-384\tau^2\Big)\\
  &-&3^3\cdot2^7\Big(1152(\frac{A'}{A})^2-384A^2-1144\tau^2\Big)\\
\nonumber  &-&3\cdot2^7\Big(2304(\frac{A'}{A})'-384(\frac{A'}{A})^2-10368A^2+312\tau^2\Big)=0
\end{eqnarray*}
that is,
$$-384\cdot2^7\cdot3\cdot32(\frac{A'}{A})^2+3^2\cdot2^7\cdot4096\tau^2=0.$$
Thus $$
\tau^2=(\frac{A'}{A})^2.$$
In a similar way  we obtain from (\ref{eq:3.78})
$$\tau\frac{A'}{A}=0.$$
The last two equations yield $\tau=0$, a contradiction, since in this case we have $K=0$, by Prop. 3.2. Hence, we have proved that there are no translation surfaces in $\R^3$ with constant non zero Gaussian curvature and this completes the proof of the Th. \ref{t1}.



\begin{thebibliography}{0}

\bibitem{da}	G. Darboux, Le\c{c}ons sur la Th\'eorie G\'en\'erale des Surfaces et ses Applications G\'eom\'etriques du Calcul Infinit\'esimal, vol. 1--4, Chelsea Publ. Co, reprint, 1972. 

\bibitem{Dillen98} F. Dillen, I. Van de Woestyne, L. Verstraelen and J. T. Walrave, The surface of Scherck in $E^3$: a special case in the class of minimal surfaces defined as the sum of two curves, Bull. Inst. Math. Acad. Sin., 26 (1998), 257--267.


\bibitem{Carmo}  M. do Carmo, Differential Geometry of Curves and Surfaces. Prentice-Hall, 1976

\bibitem{Lima} B. P. Lima, N. L. Santos and P. A. Sousa, Translation hypersurfaces with constant scalar curvature into the Euclidean space, Israel J. Math. 201 (2014), 797--811.
 
\bibitem{HLin} H. Liu, Translation surfaces with constant mean curvature in 3-dimensional spaces, J. Geom. 64 (1999), 141--149.

\bibitem{Lopez} R. L\'opez, Minimal translation surfaces in hyperbolic space, Beitr. Algebra Geom, 52 (2011), 105--112.

\bibitem{Lopez2} R. L\'opez and M. Moruz, Translation and homothetical surfaces in Euclidean space with constant curvature, J. Korean Math. Soc. 52 (2015), 523--535.

\bibitem{lomu} R. L\'opez, M. I.  Munteanu, Surfaces with constant mean curvature in Sol geometry. Differential Geom. Appl. 29 (2011), suppl. 1, S238--S245.

\bibitem{lp}  R. L\'opez and O. Perdomo,   Minimal translation surfaces in Euclidean space, J. Geom. Anal. 27 (2017), 2926--2937.


\bibitem{mm} M. Moruz, M. I. Munteanu,   Minimal translation hypersurfaces in $E^4$. J. Math. Anal. Appl. 439 (2016),  798--812. 

\bibitem{mp} M. I. Munteanu, O. Palmas,  G. Ruiz-Hern\'andez,  Minimal translation hypersurfaces in Euclidean space. Mediterr. J. Math. 13 (2016),  2659--2676. 


\bibitem{sc} H. F. Scherk, Bemerkungen \"{u}ber die kleinste Fl\"{a}che innerhalb gegebener Grenzen, J. Reine Angew. Math. 13 (1835), 185--208.
 
\bibitem{Seo} K. Seo, Translation hypersurfaces with constant curvature in space forms, Osaka J. Math. 50 (2013), 631--641.

\end{thebibliography}
\end{document}